\documentclass[12pt, reqno]{amsart}
\usepackage{amsmath, amsthm, amscd, amsfonts, amssymb, graphicx, color}
\usepackage[bookmarksnumbered, colorlinks, plainpages]{hyperref}
\hypersetup{colorlinks=true,linkcolor=red, anchorcolor=green, citecolor=cyan, urlcolor=red, filecolor=magenta, pdftoolbar=true}

\newtheorem{theorem}{Theorem}[section]
\newtheorem{lemma}[theorem]{Lemma}

\theoremstyle{definition}
\newtheorem{definition}[theorem]{Definition}

\theoremstyle{remark}

\numberwithin{equation}{section}

\begin{document}

\setcounter{page}{1}

\title[Semilinear heat equation with forcing term]{The Fujita exponent for a semilinear heat equation with forcing term on Heisenberg Group}

\author[M. B. Borikhanov, M. Ruzhansky and B. T. Torebek]{Meiirkhan B. Borikhanov, Michael Ruzhansky and  Berikbol T. Torebek$^*$}

\address{{Meiirkhan B. Borikhanov \newline Khoja Akhmet Yassawi International Kazakh--Turkish University \newline Sattarkhanov ave., 29, 161200 Turkistan, Kazakhstan \newline Institute of Mathematics and Mathematical Modeling \newline 125 Pushkin str., 050010 Almaty, Kazakhstan}}
\email{meiirkhan.borikhanov@ayu.edu.kz}
\address{{Michael Ruzhansky  \newline Department of Mathematics: Analysis, Logic and Discrete Mathematics \newline Ghent University, Belgium\newline School of Mathematical Sciences \newline Queen Mary University of London, United Kingdom}}
\email{michael.ruzhansky@ugent.be}
\address{{Berikbol T. Torebek \newline Institute of
Mathematics and Mathematical Modeling \newline 125 Pushkin str.,
050010 Almaty, Kazakhstan \newline Department of Mathematics: Analysis, Logic and Discrete Mathematics \newline
Ghent University, Belgium}}
\email{{berikbol.torebek@ugent.be}}

\thanks{$^*$ Corresponding author}



\begin{abstract} In this paper, we study a critical exponent to the semilinear heat equation with forcing term on Heisenberg group. Our technique of proof is based on methods of nonlinear capacity estimates specifically adapted to the
nature of the Heisenberg group. Surprisingly, the critical exponent will be bounded for all dimensions of the Heisenberg group, in contrast to the Euclidean case which in the 1D and 2D cases will be infinite. In addition, we give an upper bound estimate on the lifespan of local solutions.\end{abstract}
\maketitle
\tableofcontents
\section{Introduction}

We consider the following semilinear heat equation with a forcing term $f=f(\eta)$ depending on the space in the Heisenberg group
\begin{equation}\label{02}
\left\{\begin{array}{l}\large\displaystyle
u_t-\Delta_\mathbb{H}u=|u|^{p}+f,\,\ (t,\eta)\in(0,T)\times\mathbb{H}^{N},\\{}\\
u(0,\eta)=u_0(\eta),\, \eta\in\mathbb{H}^{N},\end{array}\right.\end{equation}where $p>1$ and $\Delta_\mathbb{H}$ is the sub–Laplacian operator.

At first, the problem \eqref{02} for $f\equiv0$ has studied by Pohozaev and Veron \cite{Pohozaev}, Georgiev and Palmieri \cite{ Georgiev}, and on nilpotent Lie groups by Pascucci in \cite{Pascucci}, and for general locally compact groups in \cite{Ruzhansky}. 
 
Mainly, it is proven that: 
\begin{itemize}
\item[(i)] (see \cite{Pascucci, Pohozaev}) if $1<p\leq{p}_{F}=1 + \frac{2}{Q}$, then the problem \eqref{02} admits no global positive solutions;
\item[(ii)] (see \cite{Pascucci, Georgiev}) if $p>{p}_{F}=1 + \frac{2}{Q},$ then for sufficiently small initial data, the problem \eqref{02} admits positive global solutions.
\end{itemize}
Here $Q = 2N + 2$ is the homogeneous dimension of Heisenberg group $\mathbb{H}^N$.

The above results for the subcritical ($1<p<p_F$) and supercritical ($p>p_F$) cases on Euclidean space $\mathbb{R}^N$ were first obtained by Fujita in \cite{Fujita1} and it is shown that the critical exponent is ${{p}_{F}=1+\frac{2}{N}}$.

In the critical case, $p=p_F$, this problem was studied by Hayakawa in \cite{Hayakawa} for $N=1,2$ and in \cite{Kobayashi, Weissler} for arbitrary $N$. It was shown that there is no nonnegative global solution for any nontrivial nonnegative initial data. 

The number $${{p}_{F}=1+\frac{2}{N}},$$ is called the critical Fujita exponent, which separates the nonexistence and existence of global in time solutions of semilinear heat equations.

We also note that the Fujita-type exponents to semilinear equations of type \eqref{02} on the Heisenberg groups were studied in \cite{Kirane 1, Kirane 2, Kirane 4, D'Ambrosio, D'Ambrosio1}.

Recently, the authors of \cite{Ruzhansky} have studied the following Cauchy problem on a sub-Riemannian manifold $M$
\begin{equation}\label{06}
\left\{\begin{array}{l}\large\displaystyle
u_t-\mathcal{L}_Mu=f(u),\,\ \,t>0,\,x\in M,\\{}\\
u(0,x)=u_0,\, x\in M,\end{array}\right.\end{equation}
for $u_0\geq0$ and $f(u) \simeq u^p$, where $\mathcal{L}_M$ is a sub-Laplacian on $M$.

They showed that the Fujita critical exponent is $$p_F=1+\frac{2}{D},$$
where $D$ is the global dimension for $M$ a connected unimodular Lie group.

In \cite{Bandle}, Bandle et. al. have studied the problem \eqref{02} on $\mathbb{R}^N,$ i.e
\begin{equation}\label{0.2}\left\{\begin{array}{l}
{{u}_{t}}-\Delta u={{|u|}^{p}}+f(x),\,\,\,(t,x)\in (0,\infty)\times\mathbb{R}^N , \\{}\\
u\left( 0,x \right)=u_0\left( x \right),\,\,\,x\in \mathbb{R}^N. \end{array}\right.\end{equation}
They have obtained the following results:
\begin{itemize}
\item[(i)]  If $1<p<{p}_{S}$ and $\large\displaystyle\int_{\mathbb{R}^N}f(x)dx >0$, where 
\begin{equation*} p_S=\left\{\begin{array}{l}
\infty\,\,\,\text{if}\,\,\,N=1,2, \\{}\\
\large\displaystyle\frac{N}{N-2}\,\,\,\text{if}\,\,\,N\geq3, \end{array}\right.\end{equation*}then \eqref{0.2} does not admit global in time solutions.
\item[(ii)] If $N\geq3$ and $p>{p}_{S}$, then for any $\delta>0$, there exists $\varepsilon>0$ such that \eqref{0.2} has global solutions provided that $$\max\{|f(x)|, |u_0(x)|\}\leq \frac{\varepsilon}{(1+|x|^{N+\delta})}$$ regardless of whether or not $\large\displaystyle\int_{\mathbb{R}^N}f(x)dx >0$.
\item[(iii)] If $N\geq3$, $p={p}_{S}$, $\large\displaystyle\int_{\mathbb{R}^N}f(x)dx >0$, $f(x)={O}(|x|^{-\varepsilon-N})$ as $|x|\to\infty$ for some $\varepsilon>0$, and either $u\geq0$ or 
$$\int_{|x|>R}\frac{f^-(y)}{|x-y|^{N-2}}dy=o(|x|^{-N+2}),$$
\end{itemize}
when $R$ is enough large, then \eqref{0.2} has no global solutions. Here, $f^-=\max\{-f, 0\}$.

Then, Jleli et. al. in \cite{Jleli}, have extended the results of \cite{Bandle}, by studying a semilinear parabolic equation with forcing terms depending on time and space.

The main purpose of this paper is to extend the results from \cite{Bandle} to Heisenberg groups, that is, to determine the critical exponent of problem \eqref{02}.
In particular, we prove the critical case (iii) without additional conditions on $f$. Our technique of proof is based on methods of nonlinear capacity estimates specifically adapted to the nature of the Heisenberg group. We will also give an upper bound estimate on the lifespan of local solutions.

\section{Preliminaries}
In this section let us briefly recall the necessary notions in the setting of the Heisenberg group.

The Heisenberg group $\mathbb{H}^{N}$, whose elements are $\eta=(x, y, \tau)$, where $x, y\in \mathbb{R}^N$ and $\tau\in \mathbb{R}$, is a two-step nilpotent Lie group $(\mathbb{R}^{2N+1}, \circ)$ with the group multiplication defined by
\begin{equation*}\label{HMF}
\eta\circ\eta'=(x+x', y+y', \tau+\tau'+2(\langle x,y'\rangle-\langle x',y\rangle)),
\end{equation*}
with $\langle \cdot,\cdot\rangle$ the usual inner product in $\mathbb{R}^{N}$.

The distance from $\eta=(x, y, \tau)\in \mathbb{H}^{N}$ to the origin is given by
\begin{equation}\label{HDO}|\eta|_{\mathbb{H}}=\left(\left(\sum_{i=1}^{N}(x_i^2+y_i^2)\right)^2+\tau^2\right)^{1/4}=\left(\left(|x|^2+|y|^2\right)^2+\tau^2\right)^{1/4}.\end{equation}

The dilation operation on the
Heisenberg group with respect to the group law is
\begin{equation*}\label{HD}
\delta_\lambda(\eta)=(\lambda x, \lambda y, \lambda^2\tau)\,\,\,\text{for}\,\,\,\lambda>0,    
\end{equation*}
whose Jacobian determinant is $\lambda^Q$, where $Q=2N+2$ is the homogeneous dimension of $\mathbb{H}^{N}$.

The Lie algebra $\mathfrak{h}$ of the left-invariant vector fields for $1\leq i\leq N$ on the Heisenberg group $\mathbb{H}^{N}$ is spanned by
\begin{equation*}\label{LIVF}\begin{split}
X_i:=\frac{\partial}{\partial x_i}+2y_i\frac{\partial}{\partial t},\\
Y_i:=\frac{\partial}{\partial y_i}-2x_i\frac{\partial}{\partial t},
\end{split}
\end{equation*}
with their (non-zero) commutator
\begin{equation*}\label{NZC}
[X_i,Y_i]=-4\frac{\partial}{\partial t}.    
\end{equation*}

The sub-Laplacian is defined by
\begin{equation*}\begin{split}\label{HLO}\Delta_{\mathbb{H}}&=\sum_{i=1}^N\left(X_i^2+Y_i^2\right).
\end{split}\end{equation*}

An explicit computation gives the expression
\begin{equation}\label{HLO1}\begin{split}\Delta_{\mathbb{H}}&=\sum_{i=1}^N\left(\partial^2_{x_ix_i}+\partial^2_{y_iy_i}+4y_i\partial^2_{x_i\tau}-4x_i\partial^2_{y_i\tau}+4(x_i^2+y_i^2)\partial^2_{\tau\tau}\right)\\&=\Delta_{(x,y)}+4|(x,y)|^2\partial^2_{\tau\tau}+4\sum_{i=1}^N\left(y_i\partial^2_{x_i\tau}-x_i\partial^2_{y_i\tau}\right),\end{split}\end{equation} where $\Delta_{(x,y)}$ and $|(x,y)|^2$ denote the Laplace operator and the Euclidian norm of $(x,y)$ in $\mathbb{R}^{2N}$, respectively.

The operator $\Delta_{\mathbb{H}}$ satisfies:
\begin{description}
 \item[(a)] It is invariant with respect to the left multiplication in the group: for all $\eta, \eta'\in\mathbb{H}^{N}$ we have
 $$\Delta_{\mathbb{H}}(u(\eta\circ\eta'))=\Delta_{\mathbb{H}}u(\eta\circ\eta').$$
\item[(b)] It is homogeneous with respect to the dilation in the group: for all $\lambda>0$ we have
 $$\Delta_{\mathbb{H}}(u(\lambda x, \lambda y, \lambda^2\tau))=\lambda^2(\Delta_{\mathbb{H}}u)(\lambda x, \lambda y, \lambda^2\tau).$$
  \item[(c)] Let the
function $\varphi(\eta)$ have partial symmetry with respect to origin 
$$\varphi(\eta)=\varphi(x,y,\tau)=\Tilde{\varphi}(|z|_{\mathbb{R}^{2N}},\tau),\,\,\,\,\,\eta\in \mathbb{H}^N,$$ then in view of \cite{Garofalo}
\begin{equation}\label{SS1}
\Delta_{\mathbb{H}}\varphi(x,y,\tau)=\left(\frac{\partial^2}{\partial r^2}+\frac{2N-1}{ r}\frac{\partial}{\partial r}+4r^2\frac{\partial^2}{\partial \tau^2}\right)\Tilde{\varphi}(r,\tau),\,\,\,r=|z|_{\mathbb{R}^{2N}}.
\end{equation}
\item[(d)] Moreover, if $u(\eta)=\varphi(|\eta|_{\mathbb{H}})$, then from \cite{Birindelli}, we have
\begin{equation}\label{AZQ}
\Delta_{\mathbb{H}}\varphi(r)=\frac{\sum_{i=1}^{n}(x_i^2+y_i^2)}{r^2}\biggl(\frac{d^2\varphi}{dr^2}+\frac{Q-1}{r}\frac{d\varphi}{dr}\biggr),  \end{equation}where $r=|\eta|_{\mathbb{H}}$, $Q=2N+2$ is the homogeneous dimension of $\mathbb{H}^{N}$. 
\end{description}

Since $(-\Delta_{\mathbb{H}})$ is a self-adjoint operator, we
have for all $u, v\in H^2(\mathbb{H}^N)$
\begin{equation}\label{IBP}
\int_{\mathbb{H}^N}\left[(-\Delta_{\mathbb{H}})u(\eta)\right]v(\eta)d\eta= \int_{\mathbb{H}^N}u(\eta)\left[(-\Delta_{\mathbb{H}})v(\eta)\right]d\eta.   
\end{equation}

\begin{lemma}\label{TF} Let $f\in L^1(\mathbb{H}^N)$ and $\large\displaystyle\int\limits_{\mathbb{H}^N}fd\eta>0.$ Then there exists a test function $0\leq\psi\leq1$
such that
$$\int\limits_{\mathbb{H}^N}f\psi d\eta>0.$$
\end{lemma}
\begin{proof}[Proof of Lemma \ref{TF}.] The proof will be done using the same technique as in \cite[Lemma 3.1]{Pokhozaev1} for the Heisenberg group. Here, we recall that $\eta=(x,y,\tau)\in\mathbb{R}^{2N+1}\equiv\mathbb{H}^N.$
\\Define a function  $\displaystyle\psi_R=\psi\biggl(\frac{\eta}{R}\biggr),\,0\leq\psi_R\leq1,$ such that $\psi_R(\eta)\equiv1$ for $|\eta|\leq R.$
Consequently, we have an identity in the following form
\begin{equation*}
\int\limits_{\mathbb{R}^{2N+1}}f\psi d\eta=\int\limits_{|\eta|\leq R}f\psi d\eta+\int\limits_{R\leq|\eta| }f\psi d\eta.    
\end{equation*}
Then, using the above property of the test function $\psi_R$, we obtain
\begin{equation}\label{L2}
\int\limits_{\mathbb{R}^{2N+1}}f\psi_R d\eta=\int\limits_{|\eta|\leq R}f d\eta+\int\limits_{R\leq|\eta| }f\psi_R d\eta.
\end{equation}
In view of the convergence of the integral $\displaystyle\int\limits_{\mathbb{R}^{2N+1}}|f(\eta)| d\eta$, we have
$$\biggl|\int\limits_{R\leq|\eta| }f\psi_R d\eta\biggr|\leq \int\limits_{R\leq|\eta| }|f| d\eta\to0,\,\,\text{as}\,\,\,R\to+\infty. $$
Finally, passing to the limit $R\to+\infty$ in \eqref{L2}, we deduce that
$$\lim_{R\to\infty}\int\limits_{\mathbb{R}^{2N+1}}f\psi_R d\eta=\lim_{R\to\infty}\int\limits_{R\leq|\eta| }f d\eta=\int\limits_{\mathbb{R}^{2N+1}}fd\eta>0,$$ which completes the proof, with the test function $\psi_R$ for a sufficiently large $R$.
\end{proof}

\section{Local existence of the solution}
In this section we derive the local existence of the solution.
\begin{definition}[Mild solution]\label{DMS1} A local mild solution of the problem \eqref{02} is a function $u\in C([0, T];\,L^\infty({\mathbb{H}^N}))$ such that
\begin{equation}\label{MS2}\begin{split}
u(t)=e^{-t\Delta_{\mathbb{H}}}u_0+\int_0^te^{-(t-s)\Delta_{\mathbb{H}}}(|u(s)|^p+f)ds,\end{split}\end{equation}for any $t\in[0,T).$ If $T=+\infty,$ then $u$ is a global mild solution of \eqref{02}.
\end{definition}

\begin{theorem}[Local existence]\label{T4.2} Let $p>1$ and $u_0 \in L^\infty({\mathbb{H}^N}).$  
\begin{itemize}
\item[(i)] There exists a time $0<T<\infty$ such that the problem \eqref{02} has a unique mild solution $u\in C([0, T];\,L^\infty({\mathbb{H}^N}))$.
\item[(ii)] The solution $u$ can be extended to a maximal interval $[0, T_{\max})$, where \\$0<T_{\max}\leq\infty$. Furthermore, if $T_{\max}<\infty$, then $\lim_{t\to T_{\max}^-}\|u(t)\|_{L^\infty{(\mathbb{H}^N})}=\infty$.
\item[(iii)] If $u_0, f\in L^r(\mathbb{H}^N)$, $r\in[1;\infty)$, then $$u\in C([0, T_{\max});\,L^\infty({\mathbb{H}^N}))\cap C([0, T_{\max});\,L^r({\mathbb{H}^N})).$$
\end{itemize} \end{theorem}
\begin{proof} [Proof of Theorem \ref{T4.2}.] The assertion (i).
First, for an arbitrary $T >0$, we define the Banach space $\Theta$ in the following form
\begin{equation*}\begin{split}
\Theta=\{u\in C([0, T];\,L^\infty({\mathbb{H}^N})) : \| u\|_{L^{\infty}({(0,T);\, L^{\infty}(\mathbb{H}^N}))}\leq2 \delta(u_0,f)\},\end{split}\end{equation*}where $\delta(u_0,f)=\max\{\| u_0\|_{L^{\infty}({\mathbb{H}^N})}; \|f \|_{L^{\infty}({\mathbb{H}^N})}\}.$

We endow $\Theta$ with the distance generated by the norm of  $C([0, T];\,L^\infty({\mathbb{H}^N}))$, such that
\begin{equation}\label{DIS}
d(u,v)=\|u-v\|_{L^{\infty}({(0,T);\, L^{\infty}(\mathbb{H}^N}))},\,\,\,u,v\in\Theta.     
\end{equation}
Next, for $u\in \Theta$ we define the map
$$(\Psi u)(t)=e^{-t\Delta_{\mathbb{H}}}u_0+\int_0^te^{-(t-s)\Delta_{\mathbb{H}}}(|u(s)|^p+f)ds,\,\,\,0\leq t\leq T.$$
We are going to prove the existence of a unique local mild solution as a fixed point of $\Psi$ using the Banach fixed point theorem.\\
$\bullet$ $\Psi:\Theta\rightarrow \Theta.$
Let $u\in \Theta$, then in view of \cite[Proposition 5]{Georgiev}, it implies that
\begin{align*}\| (\Psi u)(t)\|_{\Theta}&\leq\|e^{-t\Delta_{\mathbb{H}}}u_0\|_{L^{\infty}({\mathbb{H}^N})}+\biggl\|\int_0^te^{-(t-s)\Delta_{\mathbb{H}}}(|u(s)|^p+f)ds\biggr\|_{L^{\infty}({\mathbb{H}^N})}\\&\leq \| u_0 \|_{L^{\infty}({\mathbb{H}^N})}+T\left\|u\right\|^p_{L^\infty((0,T);\,L^{\infty}({\mathbb{H}^N}))}+ T\| f \|_{L^{\infty}({\mathbb{H}^N})}\\&\leq (1+T)\delta(u_0,f)
+2^pT\delta^p(u_0,f), \end{align*}
for all $0< t\leq T$. It follows that 
\begin{align}\label{LL1}\| (\Psi u)(t)\|_{\Theta}&\leq (1+T
+2^pT\delta^{p-1}(u_0,f))\delta(u_0,f). \end{align}
Let $T>0$ be sufficiently small so that
\begin{equation*}\label{TM}
T+2^pT\delta^{p-1}(u_0,f)\leq1.    
\end{equation*}
Accordingly, by \eqref{LL1}, one has 
$$\| (\Psi u)(t)\|_{\Theta}\leq2 \delta(u_0,f),$$ which gives $\Psi( \Theta)\subset\Theta$.

$\bullet$ \textbf{$\Psi$ is a contraction map.} For $u,v \in \Theta,$ we have 
\begin{align*}\| (\Psi u)(t)-(\Psi v)(t)\|_{\Theta}&=\int_0^te^{-(t-s)\Delta_{\mathbb{H}}}(|u(s)|^p-|v(s)|^p)ds\\&\leq 2^{p-1}C(p)T\delta^{p-1}(u_0,f)\left\|u-v\right\|_{L^\infty((0,T);\,L^{\infty}({\mathbb{H}^N}))}
\\&\leq \left\|u-v\right\|_{L^\infty((0,T);\,L^{\infty}({\mathbb{H}^N}))},
\end{align*}
thanks to the following inequality
\begin{equation}\label{QWE}
 ||u|^{p}-|v|^{p}|\leq C(p)|u-v|\left(|u|^{p-1}+|v|^{p-1}\right),
\end{equation}
if $T$ is chosen such that
\begin{equation*}
  2^{p-1}C(p)T\delta^{p-1}(u_0,f)<1.
\end{equation*}
The Banach fixed point theorem now ensures the existence of a mild solution to \eqref{02}.

$\bullet$ \textbf{Uniqueness of the solution.} If $u, v$ are two mild solutions in $\Theta$ for some $T > 0$, then using \cite[Proposition 5]{Georgiev} and \eqref{QWE}, we obtain
\begin{align*}\| u(s)-v(s) \|_{L^{\infty}({\mathbb{H}^N})}&=\int_0^te^{-(t-s)\Delta_{\mathbb{H}}}\||u(s)|^p-|v(s)|^p\|_{L^{\infty}({\mathbb{H}^N})}ds\\&\leq C(p)\int_0^t\left\|u(s)-v(s)\right\|_{L^{\infty}({\mathbb{H}^N})}ds.
\end{align*}
So the uniqueness follows from Gronwall's inequality.

The assertion (ii). Next, due to the uniqueness of the mild solution, we conclude that the solution of problem \eqref{02} can be extended on a maximal interval $ [0, T^{*})$, where $T^*$ is given by
$$
T^{*} = \sup\left\{ t>0 \; : \eqref{MS2}\,\, \text{admits a solution \,\,} u \in C([0,t];\,L^\infty(\mathbb{H}^N))\right\}. 
$$
Assume that $T^*<\infty$ and that there exists $K>0$ which satisfies
\begin{equation}\label{MT}
\|u(t)\|_{L^\infty(\mathbb{H}^N)}\leq K,\,\,\text{for}\,\,0\leq t<T^*.     
\end{equation}
Let $t_*$ be such that $t_*\in({T^*}/{2};T^*)$, and we define for $\tau\in(0;T^*)$ the set
$$
\Phi = \left\{v \in C([0,\tau];\, L^\infty(\mathbb{H}^N):\|v\|_{ \,L^{\infty}((0,\tau);\,L^\infty(\mathbb{H}^N))}\leq 2\delta(K,f)\right\}, 
$$where $\delta(K,f)=\max\{K,\|f\|_{L^\infty(\mathbb{H}^N)}\}$. For a given $v\in\Phi$, let us define
$$(\Lambda v)(t)= e^{-t\Delta_{\mathbb{H}}}u(t_*)+\int_0^te^{-(t-s)\Delta_{\mathbb{H}}}(|v(s)|^p+f)ds,\,\,0\leq t\leq\tau.$$
Similarly as above we endow $\Phi$ with the distance \eqref{DIS}, and we have \\$\Lambda v\in C([0,\tau];\,L^\infty(\mathbb{H}^N)).$
Moreover, from \cite[Proposition 5]{Georgiev} and \eqref{MT}, it follows that
\begin{equation}\begin{split}\label{TTT}\| (\Lambda v)(t)\|_{\Lambda}&\leq\|e^{-t\Delta_{\mathbb{H}}}u(t_*)\|_{L^{\infty}({\mathbb{H}^N})}+\biggl\|\int_0^te^{-(t-s)\Delta_{\mathbb{H}}}(|v(s)|^p+f)ds\biggr\|_{L^{\infty}({\mathbb{H}^N})}\\&\leq \| u(t_*) \|_{L^{\infty}({\mathbb{H}^N})}+\tau\left\|v\right\|^p_{L^\infty((0,T);\,L^{\infty}({\mathbb{H}^N}))}+ \tau\| f \|_{L^{\infty}({\mathbb{H}^N})}\\&\leq (1+\tau)\delta(K,f)+2^p\tau\delta^p(K,f),\end{split}\end{equation}for all $0\leq t\leq\tau.$
Let $\tau>0$ be a sufficiently small constant which satisfies
\begin{equation}\label{UND}
\tau+2^p\tau\delta^{p-1}(K,f)\leq1.    
\end{equation}
Consequently, due to \eqref{TTT} we obtain
$$\| (\Lambda v)(t)\|_{\Lambda}\leq2\delta(K,f),$$
which gives us $\Lambda(\Phi)\subset\Phi$. In addition, similiar to the argument before, under the condition \eqref{UND}, the mapping $\Lambda:\Phi\to\Phi$ is a contraction. According to the Banach contraction principle, we can see that there exists a unique function $v\in\Phi$ which satisfies
$$v(t)=e^{-t\Delta_{\mathbb{H}}}u(t_*)+\int_0^te^{-(t-s)\Delta_{\mathbb{H}}}(|v(s)|^p+f)ds,\,\,0\leq t\leq\tau.$$
At this stage, for $\max\{T_{\max}/2;T_{\max}-\tau\}<\overline{t}<T_{\max},$ let
\begin{equation*} \overline{u}(t)=\left\{\begin{array}{l}
u(t)\,\,\,\text{if}\,\,\,0\leq t\leq\overline{t}, \\{}\\
v(t-\overline{t})\,\,\,\text{if}\,\,\,\overline{t}\leq t\leq\overline{t}+\tau. \end{array}\right.\end{equation*}
Then we observe that $\overline{u}\in C([0,\overline{t}+\tau];\,L^\infty(\mathbb{H}^N))$ is a solution to \eqref{MS2} and $\overline{t}+\tau>T_{\max}$, which contradicts the definition of $T_{\max}$. Therefore, we obtain that if $T_{\max}<\infty$ then $\lim_{t\to T_{\max}^-}\|u(t)\|_{L^\infty{(\mathbb{H}^N})}=\infty$. It completes the proof.

The assertion (iii). If $u_0, f\in L^r(\mathbb{H}^N)$ with $1\leq r<\infty$, then we repeat the
fixed point argument in the space
\begin{equation*}\begin{split}
\Theta_r=&\{u\in C([0, T];\,L^\infty({\mathbb{H}^N}))\cap C([0, T_{\max});\,L^r({\mathbb{H}^N})):\\& \| u\|_{L^{\infty}({(0,T);\, L^{\infty}(\mathbb{H}^N}))}\leq2 \delta(u_0,f),\,\|u\|_{L^{\infty}((0,T);\,L^{r}({\mathbb{H}^N}))}\leq2 \delta_r(u_0,f)\},    
\end{split}\end{equation*}instead of $\Theta$, where $\delta_r(u_0,f)=\max\{\| u_0\|_{L^{r}({\mathbb{H}^N})}; \|f \|_{L^{r}({\mathbb{H}^N})}\}.$
\\In addition, we endow $\Theta_r$ with the distance 
\begin{equation*}
d(u,v)=\|u-v\|_{L^\infty((0,T);\,L^\infty(\mathbb{H}^N))}+\|u-v\|_{L^\infty((0,T);\,L^r(\mathbb{H}^N))},\,\,\,u,v\in\Theta_r.
\end{equation*}
Since
$$\||u(t)|^p\|_{L^r(\mathbb{H}^N)}\leq\|u(t)\|^{p-1}_{L^\infty(\mathbb{H}^N)}\|u(t)\|_{L^r(\mathbb{H}^N)},$$
applying the same argument as in the proof of the assertion (i), we obtain a unique solution $u$ in $\Theta_r$. Therefore, we conclude that 
$$u\in C([0, T_{\max});\,L^\infty({\mathbb{H}^N}))\cap C([0, T_{\max});\,L^r({\mathbb{H}^N})),$$which completes the proof.
\end{proof}
\begin{definition}[Weak solution]\label{DWS1} Let $u_0\in L^1_\text{loc}(\mathbb{H}^{N})$. A locally integrable function $u\in L^p_\text{loc}((0,T);\,L^p_\text{loc}(\mathbb{H}^{N}))$ is called a local weak solution of \eqref{02}, if the equality
\begin{equation}\label{WS2}\begin{split}
\int_0^T\int_{\mathbb{H}^{N}}|u|^{p}\psi  d\eta dt &+\int_{\mathbb{H}^{N}}u_0\psi(\eta,0)d\eta+\int_0^T\int_{\mathbb{H}^{N}}f\psi  d\eta dt\\&= -\int_0^T\int_{\mathbb{H}^{N}}u\psi_td\eta dt-\int_0^T\int_{\mathbb{H}^{N}}u\Delta_\mathbb{H}\psi d\eta dt,  
\end{split}\end{equation}
holds true for any function $\psi\in C^{1}_{0}((0,T);\,C^{2}_{0}(\mathbb{H}^{N})), \psi\geq0$ and $\psi(T,\cdot)=0.$ Here $C^2_0({\mathbb{H}^N})$ is a space of two-times continuously differentiable functions with compact support on $\mathbb{H}^N.$

If $T=+\infty,$ then $u$ is called a global in time weak solution of \eqref{02}.
\end{definition}

\begin{lemma}[Mild $\rightarrow$ Weak solution]\label{MWSL}
Suppose that $u_0\in L^\infty(\mathbb{H}^N)$ and let $u\in C([0,T];\,L^\infty(\mathbb{H}^N))$ be a mild solution of \eqref{02}. Then $u$ is also a weak solution to \eqref{02}.\end{lemma}

\begin{proof}[Proof of Lemma \ref{MWSL}.]
Let $T>0$ and $f\in L^1(\mathbb{H}^N)$. Then, multiplying the identity \eqref{MS2} by $\psi\in C^{1}_{0}((0,T);\,C^{2}_{0}(\mathbb{H}^{N})),$ $\psi(T,\cdot)=0$, and integrating over $\mathbb{H}^N$, we obtain
\begin{equation*}\label{MS3}\begin{split}
\int_{\mathbb{H}^N}u\psi d\eta=\int_{\mathbb{H}^N}e^{-t\Delta_{\mathbb{H}}}u_0\psi d\eta+\int_{\mathbb{H}^N}\biggl(\int_0^te^{-(t-s)\Delta_{\mathbb{H}}}(|u(s)|^p+f)ds\biggr)\psi d\eta. 
\end{split}\end{equation*}
By differentiating the last equality with respect to the variable $t$, it follows that
\begin{equation}\label{MS4}\begin{split}
\frac{d}{dt}\int_{\mathbb{H}^N}u\psi d\eta&=\int_{\mathbb{H}^N}\frac{d}{dt}\biggl(e^{-t\Delta_{\mathbb{H}}}u_0\psi \biggr)d\eta\\&+\int_{\mathbb{H}^N}\frac{d}{dt}\biggl[\biggl(\int_0^te^{-(t-s)\Delta_{\mathbb{H}}}(|u(s)|^p+f)ds\biggr)\psi \biggr]d\eta. 
\end{split}\end{equation}
Next, using \eqref{IBP}, we deduce that
\begin{equation*}\label{MS5}\begin{split}
\int_{\mathbb{H}^N}\frac{d}{dt}\biggl(e^{-t\Delta_{\mathbb{H}}}u_0\psi \biggr)d\eta&=\int_{\mathbb{H}^N}(-\Delta_{\mathbb{H}})e^{-t\Delta_{\mathbb{H}}}u_0\psi d\eta+\int_{\mathbb{H}^N}e^{-t\Delta_{\mathbb{H}}}u_0\psi_t d\eta
\\&=\int_{\mathbb{H}^N}e^{-t\Delta_{\mathbb{H}}}u_0(-\Delta_{\mathbb{H}}\psi) d\eta+\int_{\mathbb{H}^N}e^{-t\Delta_{\mathbb{H}}}u_0\psi_t d\eta, 
\end{split}\end{equation*}
and 
\begin{equation*}\label{MS6}\begin{split}
\int_{\mathbb{H}^N}\frac{d}{dt}\biggl[\biggl(\int_0^t&e^{-(t-s)\Delta_{\mathbb{H}}}(|u(s)|^p+f)ds\biggr)\psi \biggr]d\eta\\&=\int_{\mathbb{H}^N}|u(t)|^p\psi d\eta+\int_{\mathbb{H}^N}f\psi d\eta\\&+\int_{\mathbb{H}^N}\biggl(\int_0^te^{-(t-s)\Delta_{\mathbb{H}}}(|u(s)|^p+f)ds\biggr)(-\Delta_{\mathbb{H}}\psi) d\eta\\&+\int_{\mathbb{H}^N}\biggl(\int_0^te^{-(t-s)\Delta_{\mathbb{H}}}(|u(s)|^p+f)ds\biggr)\psi_t d\eta. 
\end{split}\end{equation*}
Therefore, in view of \eqref{MS2} and the last identities, the equality \eqref{MS4} can be rewritten as
\begin{equation*}\label{MS7}\begin{split}
\frac{d}{dt}\int_{\mathbb{H}^N}u\psi d\eta&=\int_{\mathbb{H}^N}u(-\Delta_{\mathbb{H}}\psi) d\eta+\int_{\mathbb{H}^N}u\psi_t d\eta+\int_{\mathbb{H}^N}|u(t)|^p\psi d\eta+\int_{\mathbb{H}^N}f\psi d\eta. \end{split}\end{equation*}

Finally, integrating in time over $[0,T]$ and noting the fact that $\psi(T,\cdot)=0$ we obtain the desired result, which completes the proof. \end{proof}

\section{Critical exponent}
In this section we investigate global properties of solutions.

\begin{theorem}\label{T1}
Let the initial data $u_0\in L^1_{\text{loc}}(\mathbb{H}^{N})$ be nonnegative and let
$$p_S=\frac{Q}{Q-2}.$$
Then we have the following properties:
\\$\bullet$ if $p\leq p_S$ and $\int_{\mathbb{H}^N}fd\eta>0$, then \eqref{02} does not admit global in time solutions.
\\$\bullet$ if $p> p_S$ and $u_0\geq0,$ $f>0$ for all $\eta\in \mathbb{H}^N$, then problem \eqref{02} admits global stationary solutions.
\end{theorem}

\begin{proof}[Proof of Theorem \ref{T1}.]
$\bullet$ {\bf Subcritical case $\large\displaystyle p<p_S=\frac{Q}{Q-2}$}. The proof is done by contradiction. Assume that there exists a global in time weak solution $u$ of the problem \eqref{02}. We need to introduce two cut-off functions. 
\\Let the functions $\mu\in C^1(0,T)$ and $\Phi\in C^2(\mathbb{R}^+)$ be such that
\begin{equation}\label{WT1}
 \mu\geq0;\,\,\,\mu\not\equiv0;\,\,\,\text{supp}(\mu)\subset(0,1)   
\end{equation}
and
\begin{equation}\label{WT2}
0\leq \Phi\leq1,\,\,\Phi\equiv 1\,\,\text{in}\,\,[0,1],\,\,\Phi\equiv0\,\,\text{in}\,\,[2,\infty).   
\end{equation}
For sufficiently large positive constant $T$, we take
\begin{equation}\label{test1}
\psi=\xi(t)\varphi(\eta)=\mu\left(\frac{t}{T}\right)\Phi\left(\frac{|\eta|^4_{\mathbb{H}}}{T^2}\right),\,\,\,(t,\eta)\in[0,T]\times \mathbb{H}^N,\end{equation}where  $|\eta|_{\mathbb{H}}$ is defined by \eqref{HDO}.
\\In addition, let $\psi$ satisfies
\begin{equation}\label{C10}\begin{split}
\int_\Sigma\psi^{-\frac{1}{p-1}}\left|\psi_t\right|^{\frac{p}{p-1}}d\eta dt<\infty, \\\int_\Omega\psi^{-\frac{1}{p-1}}\left|\Delta_\mathbb{H}\psi\right|^{\frac{p}{p-1}}d\eta dt<\infty,
\end{split}\end{equation} where $\Sigma$ contains supp$(\psi_t)$ such that 
$$\Sigma:=\left\{(t,\eta)\in \mathbb{R}_+\times\mathbb{H}^N:0\leq t\leq T,\,0\leq|\eta|^4_{\mathbb{H}}\leq 2T^2\right\}$$
and $\Omega$ contains supp$(\Delta_{\mathbb{H}}\psi)$ which defined by
$$\Omega:=\left\{(t,\eta)\in \mathbb{R}_+\times\mathbb{H}^N:0\leq t\leq T,\,T^2\leq|\eta|^4_{\mathbb{H}} \leq 2T^2\right\},$$
respectively.
\\Hence, from Definition \ref{DWS1} we have
\begin{equation}\label{N01}\begin{split}
\int_0^T\int_{\mathbb{H}^{N}}|u|^{p}\psi  d\eta dt &+\int_{\mathbb{H}^{N}}u_0\psi(\eta,0)d\eta+\int_0^T\int_{\mathbb{H}^{N}} f\psi d\eta dt\\&\leq \int_0^T\int_{\mathbb{H}^{N}}|u||\psi_t|d\eta dt+\int_0^T\int_{\mathbb{H}^{N}}|u||\Delta_\mathbb{H}\psi| d\eta dt.  
\end{split}\end{equation}
Using the $\varepsilon$-Young inequality
$$XY\leq \frac{\varepsilon}{p} X^p+\frac{1}{p'\varepsilon^{p'-1}}Y^{p'},\,\, \frac{1}{p}+\frac{1}{p'}=1,\,\, X,Y,\varepsilon>0,$$
in the right-side of \eqref{N01}, we obtain the following estimates
\begin{equation*}\label{N02}\begin{split}
\int_0^T\int_{\mathbb{H}^{N}}&|u||\psi_t|d\eta dt\leq\frac{\varepsilon}{p}\int_0^T\int_{\mathbb{H}^{N}}|u|^{p}\psi  d\eta dt+\frac{1}{p'\varepsilon^{p'-1}}\int_\Sigma\psi^{-\frac{1}{p-1}}|\psi_t|^{\frac{p}{p-1}}d\eta dt \end{split}\end{equation*}
and
\begin{equation*}\label{N03}\begin{split}
\int_0^T\int_{\mathbb{H}^{N}}&|u|\left|\Delta_\mathbb{H}\psi\right| d\eta dt\leq\frac{\varepsilon}{p}\int_0^T\int_{\mathbb{H}^{N}}|u|^{p}\psi  d\eta dt+\frac{1}{p'\varepsilon^{p'-1}}\int_\Omega\psi^{-\frac{1}{p-1}}\left|\Delta_\mathbb{H}\psi\right|^{\frac{p}{p-1}}d\eta dt.  
\end{split}\end{equation*}
Since $u_0\geq0$ and taking $\varepsilon=\large\displaystyle\frac{p}{2}$ in the last inequalities, we get
\begin{equation}\label{N04}\begin{split}
\int_0^T\int_{\mathbb{H}^{N}} f\psi d\eta dt &\leq\frac{1}{p'(\frac{p}{2})^{p'-1}}\int_\Sigma\psi^{-\frac{1}{p-1}}\left|\psi_t\right|^{\frac{p}{p-1}}d\eta dt
\\&+\frac{1}{p'(\frac{p}{2})^{p'-1}}\int_\Omega\psi^{-\frac{1}{p-1}}\left|\Delta_\mathbb{H}\psi\right|^{\frac{p}{p-1}}d\eta dt.  
\end{split}\end{equation}
Noting that $\varphi(\eta)$ is symmetric, in view of \eqref{SS1} and \eqref{WT2}, we obtain
\begin{equation*}\begin{split}
|\Delta_{\mathbb{H}}\varphi(\eta)|&=\frac{16r^6}{T^4}\varphi''\left(\frac{r^4+\tau^2}{T^2}\right)+\frac{(2N+8)r^2}{T^2}\varphi'\left(\frac{r^4+\tau^2}{T^2}\right)\\&+\frac{16r^2\tau^2}{T^4}\varphi''\left(\frac{r^4+\tau^2}{T^2}\right)+\frac{8r^2}{T^2}\varphi'\left(\frac{r^4+\tau^2}{T^2}\right)
\\&\leq\frac{C}{T}.
\end{split}\end{equation*}Accordingly, it yields that
\begin{equation}\label{KKK1}
 |\Delta_{\mathbb{H}}\psi(\eta,t)|\leq CT^{-1}\mu\left(\frac{t}{T}\right)   
\end{equation}
and
\begin{equation}\label{KKK2}
|\psi_t(\eta,t)|\leq CT^{-1}\mu_t\left(\frac{t}{T}\right)\Phi\left(\frac{|\eta|^4_{\mathbb{H}}}{T}\right).    
\end{equation}
At this stage, we will change the variables on $\Sigma$ and $\Omega$ as
\begin{equation}\label{CV}
s=\frac{t}{T}, \Tilde{x}=\frac{x}{T^\frac{1}{2}}, \Tilde{y}=\frac{y}{T^\frac{1}{2}}, \Tilde{\tau}=\frac{\tau}{T}, \,\,\Tilde{\eta}=(\Tilde{x}, \Tilde{y}, \Tilde{\tau})  
\end{equation}
and define
\begin{equation*}\begin{split}
&\Tilde{\Omega}:=\left\{(s, \Tilde{\eta})\in \mathbb{R}^+\times\mathbb{H}^N:0\leq s\leq 1,\,\,1\leq|\Tilde{\eta}|^4_{\mathbb{H}}\leq2\right\},\,\,|\Tilde{\eta}|^4_{\mathbb{H}}=\left(|\Tilde{x}|^2+|\Tilde{y}|^2\right)^2+\Tilde{\tau}^2.
\end{split}\end{equation*}
Using the inequalities \eqref{KKK1} and \eqref{KKK2} on the right-hand side of \eqref{N04} we get 
\begin{equation*}\begin{split}
\int_\Sigma\psi^{-\frac{1}{p-1}}\left|\psi_t\right|^{\frac{p}{p-1}}d\eta dt\leq CT^{-\frac{p}{p-1}+\frac{Q}{2}+1}\int_{\Tilde{\Omega}}\Phi\left(|\tilde{\eta}|^4_{\mathbb{H}}\right)\mu^{-\frac{1}{p-1}}\left(s\right)\left|\mu_s\left(s\right)\right|^{\frac{p}{p-1}}d\Tilde{\eta} ds
\end{split}\end{equation*}
and
\begin{equation*}\label{M07}\begin{split}
\int_\Omega\psi^{-\frac{1}{p-1}}\left|\Delta_\mathbb{H}\psi\right|&^{\frac{p}{p-1}}d\eta dt\leq CT^{-\frac{p}{p-1}+\frac{Q}{2}+1}\int_{\Tilde{\Omega}}\Phi^{-\frac{1}{p-1}}\left(|\tilde{\eta}|^4_{\mathbb{H}}\right)|\mu(s)|d\Tilde{\eta} ds. 
\end{split}\end{equation*}
In addition, from \eqref{test1} and \eqref{CV} it follows that
\begin{equation*}\begin{split}
\int_0^T\int_{\mathbb{H}^N} f\psi d\eta dt&=\biggl(\int_{\mathbb{H}^N} f\varphi d\eta\biggr)\biggl(\int_0^T\mu\left(\frac{t}{T}\right) dt
\biggr)\\&=T\biggl(\int_{\mathbb{H}^N} f\varphi d\eta\biggr)\biggl(\int_0^1\mu\left(s\right)ds\biggr)\\&= CT\int_{\mathbb{H}^N} f\varphi d\eta.\end{split}\end{equation*}
Combining \eqref{N04} with the last estimates we deduce that
\begin{equation*}\begin{split}
CT\int_{\mathbb{H}^N} f\varphi d\eta \leq CT^{-\frac{p}{p-1}+\frac{Q}{2}+1}.  
\end{split}\end{equation*}
Therefore, we have
\begin{equation}\label{N08}\begin{split}
\int_{\mathbb{H}^N} f\varphi d\eta \leq CT^{\lambda},  
\end{split}\end{equation}where $$\lambda:=-\frac{p}{p-1}+\frac{Q}{2}.$$
As $\large\displaystyle p<\frac{Q}{Q-2}$ we have $\lambda<0$. 

Passing to the limit  $T\to\infty$ in \eqref{N08}, we arrive at
$$\int_{\mathbb{H}^N} fd\eta\leq0,$$
which is a contradiction.

$\bullet$ {\bf Critical case $p=p_S=\large\displaystyle\frac{Q}{Q-2}$}. 
We define the test function, for $0<T, R<\infty$ sufficiently large, by 
\begin{equation}\label{test}
\psi=\mu(t)\varphi(\eta)=\mu\left(\frac{t}{T}\right)\Psi\left(\frac{\ln\left(\frac{|\eta|_\mathbb{H}}{\sqrt{R}}\right)}{\ln\left(\sqrt{R}\right)}\right),\,\,\,t>0,\,\eta\in\mathbb{H}^N.\end{equation}
\\Let $\Psi\in C^2(\mathbb{R}^+)$ be a function $\Psi:\mathbb{R}\to[0,1]$, which is the standard cut-off function defined by
\begin{equation*}
\Psi(z)=
 \begin{cases}
   1 &\text{if $-\infty<z\leq0 $},\\
   \searrow &\text{if $0<z<1$},\\
   0 &\text{if $z\geq 1$},
 \end{cases}\end{equation*} and the function $\mu(t)$ is the same as in the previous case \eqref{WT1}.
\\Suppose that $u\in L^p_\text{loc}((0,T);\, L^p_\text{loc}(\mathbb{H}^{N}))$ is a global weak solution to \eqref{02}. 

Next, following the same technique as used in the proof of the previous case, we obtain
\begin{equation}\label{CC1}\begin{split}
\int_0^T\int_{\mathbb{H}^{N}} f\psi d\eta dt &\leq C(\varepsilon)\int_\Sigma\psi^{-\frac{1}{p-1}}\left|\psi_t\right|^{\frac{p}{p-1}}d\eta dt
\\&+C(\varepsilon)\int_\Omega\psi^{-\frac{1}{p-1}}\left|\Delta_\mathbb{H}\psi\right|^{\frac{p}{p-1}}d\eta dt.  
\end{split}\end{equation}
The elementary calculations give us
\begin{equation*}\label{CC2}\begin{split}
\int_\Sigma\psi^{-\frac{1}{p-1}}\left|\psi_t\right|^{\frac{p}{p-1}}d\eta dt\leq CT\,^{1-\frac{p}{p-1}}R\,^{Q},
\end{split}\end{equation*}
\begin{equation*}\label{CC3}\begin{split}
\int_\Omega\psi^{-\frac{1}{p-1}}\left|\Delta_\mathbb{H}\psi\right|&^{\frac{p}{p-1}}d\eta dt\leq CT(\ln R)^{-\frac{Q}{2}} \end{split}\end{equation*}and
\begin{equation*}\begin{split}
\int_0^T\int_{\mathbb{H}^{N}} f\psi d\eta dt&=\int_0^T\int_{\mathbb{H}^{N}} f\varphi(\eta)\mu\left(\frac{t}{T}\right)d\eta dt=CT\int_{\mathbb{H}^N} f\varphi d\eta.\end{split}\end{equation*}
Combining the last results, we can rewrite the inequality \eqref{CC1} in the form
\begin{equation*}\begin{split}
\int_{\mathbb{H}^N} f\,\varphi d\eta &\leq C\biggl(T\,^{-\frac{p}{p-1}}R\,^{Q}+(\ln R)^{-\frac{Q}{2}}\biggr).\end{split}\end{equation*}  
Since $\large\displaystyle p=\frac{Q}{Q-2}$, it follows that 
\begin{equation*}\begin{split}
\int_{\mathbb{H}^N} f\,\varphi d\eta &\leq C\biggl(T\,^{-\frac{Q}{2}}R\,^{Q}+(\ln R)^{-\frac{Q}{2}}\biggr).\end{split}\end{equation*}
Next, for $T=R^j,\,j>0$, we have
\begin{equation}\label{CC5}\begin{split}
\int_{\mathbb{H}^N} f\,\varphi d\eta &\leq C\biggl(R\,^{Q(1-\frac{j}{2})}+(\ln R)^{-\frac{Q}{2}}\biggr).\end{split}\end{equation}
Choosing $j>3$ and passing to the limit as $R\to\infty$ in \eqref{CC5} we obtain
\begin{equation*}
\int_{\mathbb{H}^N} f\,\varphi d\eta \leq 0,\end{equation*}
which is a contradiction, in view of Lemma \ref{TF}.

$\bullet$ \textbf{The case} $\large \displaystyle p>p_S=\frac{Q}{Q-2}$. Suppose that $f(\eta)>0.$ Then the equation \eqref{02} can be rewritten as
\begin{equation*}u_t-\Delta_{\mathbb{H}}u-|u|^p>0.\end{equation*}
It is known that (see \cite{Birindelli}), when $p>\frac{Q}{Q-2}$ the equation
$$-\Delta _{\mathbb{H}}v(\eta)-|v(\eta)|^p>0,\,\,\,\eta\in\mathbb{H}^N,$$ has a positive solution.
Finally, we see that $v$ is a stationary positive super solution to the problem \eqref{02}. This completes the proof. \end{proof}
\section{Life span of local in time solutions}
In this section, we study the upper bound estimate for local in time solutions to the problem \eqref{02}.
\\Assume that $f\in L^1(\mathbb{H}^N)$ satisfies
\begin{equation}\label{D1}
f(\eta)\geq \varepsilon|\eta|_{\mathbb{H}}^{-2\lambda},\,\,\varepsilon>0,\,\,\,|\eta|_{\mathbb{H}}>1,\,\,\,\frac{Q}{2}<\lambda<\frac{p}{p-1},\end{equation}
where $|\eta|_{\mathbb{H}}$ is the distance given by \eqref{HDO}.
\begin{theorem}
Suppose that \eqref{D1} holds and $\large\displaystyle1<p<\frac{Q}{Q-2}$. Let $u_\varepsilon$ be a  local in time mild solution solution of the problem \eqref{02} on $[0,T_\varepsilon)$. Then 
$$T_\varepsilon\leq C\varepsilon^\frac{1}{\mu},\,\,\text{with}\,\,\mu=\lambda-\frac{p}{p-1}<0,$$
where $C$ is a positive constant.
\end{theorem}
\begin{proof} We take for $T_\varepsilon>0$ the test function in the following form $$\psi=\mu(t)\varphi(\eta)=\mu\left(\frac{t}{T_\varepsilon}\right)\Phi\left(\frac{|\eta|^4_{\mathbb{H}}}{T_\varepsilon^2}\right),$$ where the functions $\Phi,\mu\in C^\infty
(\mathbb{R}^+)$ is the same as  Theorem \ref{T1}. \\Therefore, instead of the estimate \eqref{N08} we have on $ [0,T_\varepsilon]\times\mathbb{H}^N$, that 
\begin{equation}\label{ZXQ}\begin{split}
\int_{\mathbb{H}^N} f\varphi d\eta &\leq CT_
\varepsilon^{-\frac{p}{p-1}+\frac{Q}{2}}.  
\end{split}\end{equation}
Next, recalling the condition \eqref{D1} and changing the variables by 
\begin{equation*}
\Tilde{x}=\frac{x}{T^\frac{1}{2}_
\varepsilon}, \Tilde{y}=\frac{y}{T^\frac{1}{2}_
\varepsilon}, \Tilde{\tau}=\frac{\tau}{T_
\varepsilon}, 
\end{equation*}
and $\Tilde{\eta}=(\Tilde{x}, \Tilde{y}, \Tilde{\tau})$, we obtain
\begin{equation}\begin{split}\label{LS02} 
\int_{\mathbb{H}^{N}}f\varphi d\eta&=T_\varepsilon^\frac{Q}{2}\int_{\mathbb{H}^N}f(\Tilde{\eta})\Phi(\Tilde{\eta})d\Tilde{\eta} 
 \\&\geq  \varepsilon T_\varepsilon^{\frac{Q}{2}-\lambda}\int_{|\Tilde{\eta}|_{\mathbb{H}}\geq \frac{1}{T_\varepsilon}}|\Tilde{\eta}|_{\mathbb{H}}^{-2\lambda}\Phi(\Tilde{\eta})d\Tilde{\eta} 
 \\&\geq  \varepsilon T_\varepsilon^{\frac{Q}{2}-\lambda}\int_{|\Tilde{\eta}|_{\mathbb{H}}\geq \frac{1}{T_0}}|\Tilde{\eta}|_{\mathbb{H}}^{-2\lambda}\Phi(\Tilde{\eta})d\Tilde{\eta},\end{split}
\end{equation}
where $T_0\in (0, T_
\varepsilon]$ is a constant, independent of $T_
\varepsilon$ and $\varepsilon$.
\\Therefore, using \eqref{LS02} and \eqref{ZXQ} we arrive at
\begin{equation}\label{L1}\begin{split}
\varepsilon \leq CT_
\varepsilon^{-\frac{p}{p-1}+\lambda} .  
\end{split}\end{equation}
Since $\large\displaystyle\lambda<\frac{p}{p-1}$, it follows that $\large\displaystyle\mu=\lambda-\frac{p}{p-1}<0.$

Finally, we obtain $T_\varepsilon\leq C\varepsilon^\frac{1}{\mu},$
with a constant $C>0$.\end{proof}

\section*{Acknowledgments}
The authors would like to thank Durvudkhan Suragan and Bharat Talwar for their constructive comments
that helped to improve the quality of the manuscript.

\section*{Funding}
This research has been funded by the Science Committee of the Ministry of Education and Science of the Republic of Kazakhstan (Grant No. AP14869090) and by the FWO Odysseus 1 grant G.0H94.18N: Analysis and Partial Differential Equations and by the Methusalem programme of the Ghent University Special Research Fund (BOF) (Grant number 01M01021). Michael Ruzhansky is also supported by EPSRC grant EP/R003025/2. No new data was collected or generated during the course of research.

\end{document}